\documentclass[american]{amsart}

\usepackage[T1]{fontenc}
\usepackage[utf8]{inputenc}
\usepackage[english]{babel}

\usepackage{amsfonts}
\usepackage{amssymb}
\usepackage{amsthm}
\usepackage{amsmath}
\usepackage[all]{xy}
\usepackage{graphicx,psfrag}
\usepackage{hyperref}
\usepackage{svgcolor}

\everymath{\displaystyle}

\newtheorem{theo}{Theorem}

\newtheorem{corollaire}[theo]{Corollary}
\newtheorem{lemme}[theo]{Lemma}

\newtheorem{prop}[theo]{Property}

\theoremstyle{definition}
\newtheorem{definition}[theo]{Definition}
\newtheorem{def-theo}[theo]{Definition-Property}
\theoremstyle{definition}

\newtheorem{remarque}[theo]{Remark}
\newtheorem{question}[theo]{Question}
\newtheorem{remarques}[theo]{Remarks}

\DeclareMathOperator{\ii}{\textit{i}}
\DeclareMathOperator{\id}{id}
\DeclareMathOperator{\ext}{Ext}

\DeclareMathOperator{\Mod}{Mod}

\DeclareMathOperator{\dd}{\mathbb{D}}

\DeclareMathOperator{\tei}{\mathcal{T}\!\left( X_0 \right)}
\DeclareMathOperator{\mf}{\mathcal{M}\mathcal{F}}
\DeclareMathOperator{\pmf}{\mathcal{P}\mathcal{M}\mathcal{F}}
\DeclareMathOperator{\ray}{\mathcal{R}}
\DeclareMathOperator{\rs}{\mathbb{R}_{\geq 0}^{\mathcal{S}}}
\DeclareMathOperator{\prs}{\textrm{P}\mathbb{R}_{\geq 0}^{\mathcal{S}}}
\DeclareMathOperator{\teith}{\overline{\tei}^{\textrm{\scriptsize{Th}}}}
\DeclareMathOperator{\teigm}{\overline{\tei}^{\textrm{\scriptsize{GM}}}}
\DeclareMathOperator{\teigmb}{\partial_{\textrm{\scriptsize{GM}}}\!\!\tei}
\DeclareMathOperator{\cvgth}{\overset{{\scriptscriptstyle \textrm{Th}}}{\longrightarrow}}
\DeclareMathOperator{\cvggm}{\overset{\scriptscriptstyle \textrm{GM}}{\longrightarrow}} 
\DeclareMathOperator{\horo}{\mathcal{H}}
\DeclareMathOperator{\sl2}{\textrm{SL}_{2}\!\left( \mathbb{R}\right)}
\DeclareMathOperator{\so2}{\textrm{SO}_{2}\!\left( \mathbb{R}\right)}
\DeclareMathOperator{\qd}{\mathcal{Q}}
\DeclareMathOperator{\nlongrightarrow}{\not\hspace{-4pt}\longrightarrow}

\newcommand{\quot}[2]{{\left. \raisebox{1.5px}{$#1$}\middle/ \raisebox{-2px}{$#2$}\right.}}

\renewcommand{\epsilon}{\mathcal{E}}

\begin{document}

\title[Convergence of some horocyclic deformations]{Convergence of some horocyclic deformations to the Gardiner-Masur boundary}
\author[Vincent Alberge]{Vincent Alberge}
\thanks{The author is partially supported by the French ANR grant FINSLER. The author is also partially supported by the U.S. National Science Foundation grants DMS 1107452, 1107263, 1107367 "RNMS: GEometric structures And Representation varieties" (the GEAR Network).}

\begin{abstract}
We introduce a deformation of Riemann surfaces and we are interested in the convergence of this deformation to a point of  the Gardiner-Masur boundary of  Teichmüller space. This deformation, which we call the horocyclic deformation, is directed by a projective measured foliation and  belongs to a certain horocycle in a Teichmüller disc. Using works of  Marden and  Masur in \cite{marden&masur} and Miyachi in \cite{miyachi1, miyachi2, miyachigromov}, we show that the horocyclic deformation converges if its direction is given by a simple closed curve or a uniquely ergodic measured foliation.
\end{abstract}


\maketitle

\section{Introduction}

In this paper, all Riemann surfaces considered are conformal structures on a closed connected surface of finite topological type $\left( g, n\right)$, where $g$ represents the genus and $n$ the number of marked points. We shall assume that for such a Riemann surface $X$, the corresponding Euler charasteristic  $\chi \left( X \right)= 2-2g -n$ is strictly negative. By the uniformization theorem, this implies that $X$ is endowed with a unique conformal metric of constant curvature $-1$. 

To a fixed conformal structure $X_0$, we can associate the Teichmüller space $\tei$, which classifies in some sense the different conformal structures with the same topological type as $X_0$.  The more precise definition will be recalled in Subsection \ref{secttei}. This definition leads to the conformal point of view of the Teichmüller space. There exists an equivalent point of view on this space which uses the uniformization theorem, namely the hyperbolic one. It allows to define $\tei$ as the set of all hyperbolic metrics defined on the underlying surface of $X_0$ up to isometries isotopic to the identity.  Depending on the point of view we use, there are two natural tools, the hyperbolic length, for the hyperbolic point of view and the extremal length, for the conformal point of view. These two tools lead to two well-known metrics (one of them is asymmetric) and also respectively two compactifications, the Thurston one and the Gardiner-Masur one.  There are many deformations of conformal structures or of hyperbolic structures, which may be geodesic or not with respect to the metric used. 

Using the hyperbolic point of view, we can consider two natural deformations in the Teichmüller space, stretches  and earthquakes. The stretch line is directed by a complete geodesic lamination on a hyperbolic surface and defines a geodesic line with respect to the so-called Thurston asymmetric metric on $\tei$. For any complete geodesic lamination, and so for any stretch line, we can associate a measured foliation which is usually called the stump of the given direction. Théret showed in \cite{theret2} that such a line converges in the reverse direction in the Thurston boundary to the projective class of the stump, if the stump is either a simple closed curve or a uniquely ergodic foliation. Moreover, Théret also solved in \cite{theret3} the negative convergence of stretch line whose the associated stump is a rational measured foliation. The limit being the barycenter of the stump. Note that the convergence in the positive direction has been  solved by Papadopoulos in \cite{papadop}. 
The earthquake deformation, which was introduced by Thurston and which generalizes the Fenchel-Nielsen deformation,  is directed by a measured foliation class. It is well known that the hyperbolic length of the direction remains invariant along this deformation. Moreover, the earthquake converges  to the projective class of the corresponding direction in the Thurston boundary. The earthquake deformations are not geodesic as the stretch lines are, but they have a natural behaviour relative to each other. Indeed, under some assumption on directions, Théret showed in \cite{theret} that doing first a stretch  and after an earthquake  is the same  as doing first an earthquake and after a stretch. In this paper, we shall keep in mind all these properties and consider other deformations and their convergence in the Thurston boundary and in the Gardiner-Masur boundary.. 

From the conformal point of view of $\tei$, there exists a well-known conformal deformation which is called here Teichmüller deformation and which plays the role of stretch line in this point of view. This deformation, which is  directed by a  projective class of measured foliation, is geodesic with respect to the Teichmüller metric, and some investigations about the convergence in these two compactifications have already been done. Indeed, in the case of Thurston's compactification, it is well known that a Teichmüller deformation directed by a simple closed curve converges to its direction. Masur showed in \cite{masurcommute} that this is also the case if the direction is uniquely ergodic. Later, Lenzhen in \cite{lenzhen} constructed an example of such a deformation which does not converge in this compactification. However, in the Gardiner-Masur compactification, the convergence is most natural. Liu and Su in \cite{liu&su} and Miyachi in \cite{miyachi2}, proved that any Teichmüller deformation converges. Using Kerckoff's computations, Miyachi also gave in \cite{miyachi1} the explicit limit when the direction is given by a rational measured foliation. Walsh in \cite{walsh2} generalized this result by giving the limit for any direction. 

In this paper, we shall be interested in another deformation, which we called horocyclic deformation, and which will be the natural analogue of the earthquake deformation from the conformal point of view. The horocyclic deformation is directed by a projective class of measured foliation and  stays in a Teichmüller disc. This  deformation also presents a nice behaviour with respect to the Teichmüller deformation. Indeed,  seen as maps, these two deformations commute if the directions are the same.  We will see that the extremal length of the direction stays invariant along the associated horocyclic deformation. We will also see that this deformation converges in the Gardiner-Masur compactification if the given direction is either a simple closed curve or a projective class of a uniquely ergodic foliation. In these two cases, we shall give the limits and remark that they correspond to limits of associated Teichmüller's deformations. We will also see that in these two particular cases the horocyclic deformations converge also in the Thurston compactification. In contrast with the hyperbolic point of view, we will see through an example that there exists a horocyclic deformation which does not converge to the same limit as the corresponding Teichmüller deformation. 

\sloppy
\subsection*{Acknowledgments} The author wants to express his gratitude to Professors Ken'ichi Ohshika and Athanase Papadopoulos for their useful comments and suggestions. The author would like to thank Professor Hideki Miyachi for his interest and for his help in the proof of Theorem \ref{theoergodic}. The author also thanks Daniele Alessandrini, Maxime Fortier Bourque and Weixu Su for useful conversations.

The author acknowledges the hospitality of the Graduate Center of  CUNY where a part of this work was done. The author wishes to thank especially Professor Ara Basmajian for his kindness during this stay.
\fussy

\section{Backgrounds}

\subsection{Teichmüller space}\label{secttei}

Let $X_0$ be a fixed Riemann surface. We say that $\left( X_1 , f_1 \right)$ and $\left( X_2 , f_2 \right)$, where $f_i : X_0 \rightarrow X_i$ ($i=1, \,2$) is a quasiconformal homeomorphism, are equivalent if there exists a conformal map $h: X_1 \rightarrow X_2$ which is homotopic to $f_2 \circ f_1 ^{-1}$. The \textit{Teichmüller space} of $X_0$, denoted by $\tei$, is the set of equivalence classes of pairs $\left( X, f \right)$. For a pair $\left( X, f \right)$, we denote the corresponding point in $\tei$  by $\left[ X, f \right]$. We call $x_0 =\left[ X_0 , \id\right]$ the base point of $\tei$. There is a natural distance on $\tei$, called the \textit{Teichmüller distance} defined as the follows. Let $x=\left[ X, f \right]$ and $y=\left[ Y, g \right]$ be two points in $\tei$. The Teichmüller distance between $x$ and $y$ is
\begin{equation}\label{defmetrique}
d_{T}\left( x, y \right)= \inf\log{K_h},
\end{equation}
where $h$ is taken  over all quasiconformal homeomorphisms homotopic to $g\circ f^{-1}$ and $K_h$ denotes the \textit{quasiconformal dilatation} of $h$. We recall that the quasiconformal dilatation of $h$ is defined as the essential supremum of 
$$
p\in X_1 \mapsto \frac{\vert \partial\!_{\overline{z}} h\left( z \right) \vert + \vert \partial\!_{z} h\left( z \right)\vert}{\vert \partial\!_{z} h\left( z \right)\vert -\vert \partial\!_{\overline{z}} h\left( z \right)\vert},
$$
where for local coordinates $z=x+\ii y$,
\begin{align*}
\begin{cases}
\partial\!_{z} h &= \frac{1}{2}\left( \frac{\partial h}{\partial x} - \ii \frac{\partial h}{\partial y}\right) ,\\
\partial\!_{\overline{z}} h &= \frac{1}{2}\left( \frac{\partial h}{\partial x} + \ii \frac{\partial h}{\partial y}\right).
\end{cases}
\end{align*}

\subsection{Measured foliations}

We say that a simple closed curve on $X_0$ is \textit{essential} if it is not homotopic to a point. We denote the set of homotopy classes of essential simple closed curves by $\mathcal{S}\left(X_0 \right)$ or by $\mathcal{S}$ if there is no confusion. Given two elements $\alpha$ and $\beta$ of $\mathcal{S}$, we define their \textit{geometric intersection number}, denoted by $i\left( \alpha, \beta \right)$, as the minimal intersection number of two essential simple closed curves in their homotopy classes. To simplify notation, we call $\mathcal{S}$ the set of simple closed curves.

We set $\mathbb{R}_{+}\times\mathcal{S}=\left\lbrace t\cdot\alpha \,\mid\, t\geq 0 \textrm{ and } \alpha\in\mathcal{S} \right\rbrace$ and we call it the set of \textit{weighted simple closed curves}. It is known that 
\begin{align*}
i_{\star} : \mathbb{R}_{+}\times\mathcal{S} &\rightarrow \mathbb{R}_{+}^{\mathcal{S}}\\
t\cdot \alpha &\mapsto t\cdot i\left( \alpha , \cdot \right)
\end{align*}
is an embedding. We denote the closure of the image of this map by $\mf$ and we call it the set of \textit{measured foliations}. We define the space $\pmf$ of \textit{projective measured foliations}  as the quotient of $\mf \setminus \left\lbrace 0\right\rbrace$ by the natural action of  $\mathbb{R}_+$. We denote by $\left[ F\right]$ the projective class of $F\in\mf$. It is well known that $\mf$ and $\pmf$ are respectively homoemorphic to  $\mathbb{R}^{6g-6+2n}$ and $\mathbb{S}^{6g-7+2n}$. Furthermore, it is known that the geometric intersection number can be extended to a continuous function from $\mf\times \mf$ to $\mathbb{R}_{+}$. Thus, any measured foliation $F$ can be seen as a continuous function from $\mf$ to $\mathbb{R}_{+}$. We refer to \cite{flp} for a more geometric interpretation. This interpretation allows to introduce the notion that we not recall of critical points and critical graph associated to a measured foliation. 

A measured foliation $F$ is \textit{rational} if it is determined by a system of positive real numbers $\left\lbrace w_k \right\rbrace_{1\leq i \leq k}$ and a system of distinct simple closed curves $\left\lbrace \alpha_i \right\rbrace_{1\leq i \leq k}$ such that 
$$
\forall G\in\mf, \, i\left( F , G \right) = \sum_{1\leq i \leq k}{w_i i\left( \alpha_i , G\right)}.
$$

A measured foliation $F$ is \textit{uniquely ergodic} if for any $G\in\mf$ such that $i\left( F, G \right)=0$ then $F$ and $G$ are projectively equivalent.

A pair $\left( F, G \right)$ of measured foliations is said to be \emph{transverse}  if for any $H\in\mf-\left\lbrace 0 \right\rbrace$, $i\left( F, H \right)+ i\left( G , H \right)>0$.

\subsection{Quadratic differentials}

A \textit{regular quadratic differential} $q$ on $X$ is locally the data of $q=q\left( z \right) dz^2$ such that $q\left( z \right)$ is meromorphic with poles of order at most $1$ at the marked points. Such a quadratic differential determines a pair of transverse measured foliation, $F_{v, q}$ and $F_{h,q}$ which are respectively called the \textit{vertical foliation} and the \textit{horizontal foliation} of $q$ on $X$. The critical points of these foliations correspond to zeros or poles of $q$. We denote by $\qd\left( X \right)$ the space of such a quadratic differentials on $X$. This space is equipped with an $L^1$-norm $\| \cdot \|$which is defined as follows. For any $q\in\qd\left( X \right)$, 
$$
\| q \| = \iint_{X}{\vert q \vert}.
$$
We set $\qd_{1}\left( X \right)$, the set of elements of $\qd\left( X \right)$ which are of norm $1$. 
Furthermore, for any $q\in\qd\left( X \right)$, there exists a system of local coordinates $z=x+\ii y$ on $X$ where away from the critical points of $q$, we have $q=dz^2$.  Such a system is called \emph{natural coordinates} of $q$.

It is well known that  a transverse pair of measured foliation  $\left( F, G \right)$ defines a  Riemann surface $X$ and a regular quadratic differential $q$ on $X$ where $F$ (resp. $G$) corresponds to the horizontal (resp. vertical) foliation of $q$. Such a pair determines also a point in the Teichmüller space. 

Another link between quadratic differentials and measured foliations is given by the following result of Hubbard and Masur.
\begin{theo}[\cite{hubbard&masur}]\label{hubbard&masur}
Let $X$ be a Riemann surface and $F\in\mf$. Then there exists a unique regular quadratic differential $q$ on $X$ such that $F_{h,q}=F$. 
\end{theo}

\begin{remarques}\label{remarquehubbard}\begin{enumerate}
\item We will use Theorem \ref{hubbard&masur} in the following form. Let $x=\left[X, f \right]\in\tei$ and $F\in\mf$. Then there exists a unique regular quadratic differential $q$ on $X$ such that $F_{h, q}=f_{\ast}\!\left( F \right)$. We will denote such a quadratic differential by $q_{x, F}$ or $q_{F}$ if there is no confusion. 
\item Hubbard and Masur (and Kerckhoff in \cite{kerckhoff}) proved a stronger result which says that $\mf$ is homeomorphic to $\qd\left( X \right)$ when we consider  these two spaces with the topology respectively induced by, the geometric intersection and the norm $\| \cdot \|$  . An equivalent statement is that $\pmf$ is homeomorphic to $\qd_{1}\left( X \right)$.
\end{enumerate}
\end{remarques}

\subsection{Extremal length}

For a \textit{measurable conformal metric} $\rho=\rho\left( z \right) \vert dz\vert$ on a Riemann surface $X$, we set $A_{\rho}=\iint_{X}{\rho^2}$. Furthermore, for $\alpha\in\mathcal{S}\left( X\right)$ we set $L_{\rho}\left( \alpha\right)=\inf_{\alpha^\prime}{\int_{\alpha^\prime}{\rho\left( z \right) \vert dz \vert}}$ where $\alpha^\prime$ belongs to $\alpha$.  The \textit{extremal length} of $\alpha$ on $X$ is defined as
\begin{equation}\label{defanal}
\ext_{X}\left( \alpha\right) = \sup_{\rho}{\frac{L_{\rho}\left(\alpha\right)^2}{A_{\rho}}},
\end{equation}
where $\rho$ ranges over all measurable conformal metrics such that $A_\rho \neq 0, +\infty$. We refer to \cite{ahlfors} for an introduction to this notion. This definition is called the \emph{analytic definition} of extremal length.

There exists an equivalent definition of extremal length which is as follows. Let $\alpha\in\mathcal{S}$. Then the extremal length of $\alpha$ on $X$ is 
\begin{equation}\label{defgeo1}
\ext_{X}\left( \alpha \right) = \inf \frac{1}{\Mod\left( A \right)},
\end{equation}
where $A$ ranges over all Euclidean cylinders which can be conformally embedded into $X$ and whose  image of the core curve by this embedding  is in the class $\alpha$. We recall that the modulus of a Euclidean cylinder is the ratio between the height and the circumference.

As the extremal length is a conformal invariant, we can naturally define it on Teichmüller space. Indeed, let $x=\left[ X,f \right]\in\tei$ and $\alpha\in\mathcal{S}$, we set $\ext_{x}\left( \alpha\right)=\ext_{X}\left( f\left(\alpha\right)\right)$. 

Moreover, by setting for any $t\geq 0$, $\ext_{x}\left( t\cdot\alpha\right)=t^2 \ext_{x}\left( \alpha\right)$, Kerckhoff showed in \cite{kerckhoff} that $\ext_{x}\left( \cdot \right)$ extends continuously to $\mf$ and 
\begin{equation}\label{defgeo}
\ext_{x}\left( F\right)=\| q_{x, F} \|,
\end{equation}
where $q_{x, F}$ is the unique quadratic differential on $X$ whose horizontal foliation is  $f_{\star}\left( F \right)$. Kerckhoff proved also a result now called \textit{Kerckhoff's formula}.
\begin{theo}[Theorem 4, \cite{kerckhoff}]
Let $x$, $y\in\tei$. Then 
\begin{equation}\label{kerckhoff}
d_T \left( x,y \right) = \log\sup_{\alpha\in\mathcal{S}}{\frac{\ext_{y}\left( \alpha\right)}{\ext_{x}\left( \alpha\right)}}.
\end{equation}  
\end{theo}
We set  
$$
\mf_{1}=\left\lbrace F\in\mf \,\mid \, \ext_{x_{0}}\left( F \right)= 1 \right\rbrace.
$$
It can be shown that $\mf_{1}$ is homeomorphic to $\pmf$.

\subsection{Teichmüller's theorem}

Before presenting the Teichmüller theorem, we recall that for any projective measured foliation $\left[ F\right]$ and $0\leq k <1$, there exists a unique quasiconformal map $f_{k}^{\left[ F\right]} : X_0 \rightarrow f_{k}^{\left[ F\right]}\left( X_0\right)$ which is the solution of the \textit{Beltrami equation}
\begin{equation}\label{beltrami}
\partial_{\bar{z}}f =-k \frac{\overline{q_{F}}}{\vert q_{F} \vert}\partial_{z}f.
\end{equation}
 Let us recall that $f_{k}^{\left[ F\right]}$ is called the \textit{Teichmüller map} associated to $q_F$ and whose  quasiconformal dilatation is equal to $\frac{1+k}{1-k}$.

Let $x=\left[ X, g\right]\in\tei$. For any $t\geq 0$ and for any $F\in\mf$, we set 
$$
\ray_{\left[ F\right]}^{t}\left( x \right)=\left[ f_{\tanh\left( \frac{t}{2} \right)}^{\left[ F\right]}\left( X \right), f_{\tanh\left( \frac{t}{2}\right)}^{\left[ F\right]}\circ g \right].
$$ 
The Teichmüller theorem, which was first proved for closed surfaces by Teichmüller in \cite{teichmuller2} and after by Bers for most general cases in \cite{bers},  says that 
\begin{align}\label{theoremteich}
\mathbb{R}_{+}\times \pmf &\rightarrow \tei \nonumber \\
\left( t, \left[ F \right]\right) &\mapsto \begin{cases}
\ray_{\left[ F\right]}^{t}\left( x \right) &\textrm{ if } t>0 \\
0 &\textrm{ if } t=0
\end{cases}
\end{align}
is a homeomorphism.

Moreover, Teichmüller proved in \cite{teichmullermain}  that  $t\mapsto \ray_{\left[ F\right]}^{t}\left( x \right)$ is a geodesic ray (with respect to the Teichmüller distance) parametrized by arc length. We call $\left(\ray_{\left[ F\right]}^{t}\left( x \right)\right)_{t\geq 0}$ the \textit{Teichmüller ray} emanating from $x$ and directed by $\left[ F\right]$. We will use also the term \emph{Teichmüller deformation}. If $G$ represents the vertical foliation of $q_F$, then the set  $\left( \ray_{\left[ G \right]}^{t}\left( x \right)\right)_{t\geq 0} \cup \left( \ray_{\left[ F\right]}^{t}\left( x \right)\right)_{t\geq 0}$ forms a geodesic line called the \emph{Teichmüller geodesic line} through $x$ and directed by $\left( F, G \right)$. By abuse of notation we will denote it by $\left( \ray_{\left[ F\right]}^{t}\left( x \right)\right)_{t\in\mathbb{R}}$. 

Furthermore, the Teichmüller map $f_{\tanh\left( \frac{t}{2}\right)}^{\left[ F\right]}$ determines a regular quadratic differential $q_{t}$ on $f_{\tanh\left( \frac{t}{2} \right)}^{\left[ F\right]}\left( X \right)$ such that $F_{h,q_t}=e^{\frac{t}{2}}\cdot F$ and $F_{v,q_t}=e^{-\frac{t}{2}}\cdot F_{h,q_{F}}$. Thus, we have

\begin{align}\label{rel}
\begin{cases}
\ext_{\ray_{\left[ F\right]}^{t}\left( x \right)}\left( F\right) &=e^{-t}\ext_{x}\left( F\right),\\
\ext_{\ray_{\left[ F\right]}^{t}\left( x \right)}\left( F_{v,q_F} \right) &=e^{t} \ext_{x}\left( F_{v,q_F} \right).
\end{cases}
\end{align}

\subsection{Compactifications of Teichmüller space}

There exist several different compactifications of Teichmüller space which depend on which point of view we use. In this paper, we are interested in two of them: the \textit{Thurston compactification} and the \textit{Gardiner-Masur compactification}. As we shall see below, these compactifications are constructed in  similar ways. We recall that for each point $x$ in $\tei$, the hyperbolic length on $x$ of an element $\alpha\in\mathcal{S}$ is well defined. We denote this length by $l_{x}\left(\alpha\right)$. Thus, $l_{x}\left( \cdot \right)$ determines an element of $\rs$ and so we can define
\begin{equation}\label{thurston}
\Phi_{\textrm{Th}} : x\in\tei\mapsto \left[ l_{x}\left( \cdot \right)\right] \in \prs,
\end{equation}
where $\prs = \quot{\rs\setminus\left\lbrace 0 \right\rbrace}{\mathbb{R}_{>0}}$. Thurston showed that $\Phi_{\textrm{Th}}$ is an embedding whose  image is relatively compact. We denote the closure of this image by $\teith$ and we call it the Thurston compactification of $\tei$. An important fact is that the boundary of the closure  is exactly $\pmf$. We refer to \cite{flp} for more details. 

In the same way, we define 
\begin{equation}\label{gm}
\Phi_{\textrm{GM}} : x\in\tei\mapsto \left[ \ext_{x}^{\frac{1}{2}}\left( \cdot \right)\right] \in \prs.
\end{equation}
Gardiner and Masur showed in \cite{gm} that $\Phi_{\textrm{GM}}$ is also an embedding with relatively compact image. The closure  of the image is denoted by $\teigm$ and called the Gardiner-Masur compactification of $\tei$. Gardiner and Masur also showed  that if $\dim_{\mathbb{C}}\tei = 3g-3+n \geq 2$, then $\pmf \subsetneq \teigmb$. Miyachi proved in \cite{miyachiproc} that in the case of the once-punctured torus,  these two boundaries are the same. He also showed in \cite{miyachi1} that each point in the Gardiner-Masur boundary can be represented by a continuous function from $\mf$ to $\mathbb{R}_{+}$. Indeed, if we set for each $y\in\tei$,
\begin{equation}
\epsilon_{y} : F\in\mf \mapsto \left( \frac{\ext_{y}\left( F\right)}{e^{d_{T}\left( x_0 , y \right)}} \right)^\frac{1}{2},
\end{equation}
then  $\left(\epsilon_{y}\right)_{y\in\tei}$ forms a normal family and  we have the following statement.
\begin{theo}[\cite{miyachi1}, Theorem 1.1]\label{pointbordgm}
Let $p\in\teigmb$. Then there exists a unique continuous map $\epsilon_{p} : \mf \rightarrow \mathbb{R}_{+}$ such that
\begin{enumerate}
\item $\epsilon_{p}$ represents the point p,
\item $\max_{F\in\mf_1}{\epsilon_{p}\left( F\right)}=1$,
\item if $y_{n}$ converges to $p$, then $\epsilon_{y_n}$ converges uniformly to $\epsilon_p$  on any compact set of $\mf$.
\end{enumerate}
\end{theo}

For more details about this compactification, we refer to \cite{miyachisurvey}.

In order to distinguish the convergence in these two compactifications, we write $\cvgth$ and $\cvggm$. Furthermore, from the Kerckhoff formula and Theorem \ref{pointbordgm}, we get the following.

\begin{lemme}\label{lemmebidon}
Let $\left( y_n \right), \left( z_n \right) \subset \tei$ such that $y_n \cvggm p$ and $z_n \cvggm q$. Then 
$$
d_{T}\left( y_n , z_n \right) \rightarrow 0 \Longrightarrow p=q.
$$
\end{lemme}

\begin{proof}
If $p$ or $q$ belongs to $\tei$ then the proof is obvious. Suppose that $p, q\in\teigmb$. According to Miyachi's result (Theorem \ref{pointbordgm}), it suffices to show that $\epsilon_{p}=\epsilon_{q}$.  

By Kerckhoff's formula we have for all $n\in\mathbb{N}$ and for all $\alpha\in\mathcal{S}$,
$$
0\leq \epsilon_{x_n}^{2}\left( \alpha\right)\leq \epsilon_{y_n}^{2}\left( \alpha\right)e^{d_{T}\left( x_n , y_{n}\right)}e^{d_{T}\left( x_0 , y_n \right) -d_{T}\left( x_0 , x_n \right)}.
$$
Thus, when $n$ tends to $+\infty$ we obtain
$$
\epsilon_{p}\left( \alpha \right) \leq \epsilon_{q}\left( \alpha\right).
$$
As the Teichmüller distance is symmetric, we have the reverse inequality and so for any $\alpha\in\mathcal{S}$, 
$$
\epsilon_{p}\left(\alpha\right) = \epsilon_{q}\left( \alpha\right).
$$
The lemma is now proved.
\end{proof}

If we only suppose in Lemma \ref{lemmebidon} that the distance is bounded, we do not necessarily have $p=q$. Indeed, Masur proved in \cite{masurcontreexemple}  that two Teichmüller rays starting at the same point are bounded if they are directed by topologically the same rational foliation. But Miyachi gave in \cite{miyachi1} the limit of such a rays in the Gardiner-Masur boundary. In particular, when the foliations are given by at least two disjoint simple closed curves, then the limits are distinct if the foliations are not projectively equivalent. An explicit expression of such a limit is given by Relation (\ref{limitrationnel}) below.  

Moreover, in this form, the converse of this lemma is wrong. It suffices to set for any measured foliation $F$, $y_n = \ray_{\left[F \right]}^{n^2}\left( x_0 \right)$ and $z_n = \ray_{\left[F \right]}^{n}\left( x_0 \right)$. Note that Liu and Su proved in \cite{liu&su} that any Teichmüller ray converges in the Gardiner-Masur boundary. We also refer to \cite{miyachi2} and \cite{miyachisurvey}. 

However, about the converse of this lemma, we can ask:
\begin{question}\label{question1}
Can we find two sequences $\left(y_n \right)$ and $\left( z_n \right)$ in $\tei$ such that
\begin{itemize}
\item $y_n \underset{n\rightarrow +\infty}{\cvggm} p$ and $z_n \underset{n\rightarrow +\infty}{\cvggm} p$, where $p\in\teigmb$,
\item $d_{T}\left( y_n , x_0 \right) / d_{T}\left( z_n , x_0 \right) \underset{n\rightarrow +\infty}{\longrightarrow} 1$,
\item $d_T \left( y_n , z_n \right) \underset{n\rightarrow +\infty}{\nlongrightarrow} 0$?
\end{itemize} 
\end{question}
A positive answer will be given below by considering  Teichmüller discs.

\section{Horocyclic deformation}

\subsection{Teichmüller disc}

We start this subsection by recalling the notion of  Teichmüller discs and their known propertiest. Let $x= \left[ X, f \right]\in\tei$ and $F\in\mf$. By Theorem \ref{hubbard&masur} and Remark \ref{remarquehubbard}, we can associate to $F$ a unique regular quadratic differential $q$ on $X$ whose horizontal foliation is $f_{\ast}\left( F \right)$.  It is well known that 
\begin{align}\label{teichmullerdisc}
\imath_{\left( x, \left[ F \right]\right)} \, :\,  \mathbb{D} &\rightarrow \tei \\ r\cdot e^{\ii \theta}&\mapsto \ray_{\left[ F_{h, e^{-\ii\theta}q}\right]}^{2\tanh^{-1}\left( r \right)}\left( x \right)
\end{align}
is an isometric embeding, when we consider the Poincaré metric on $\dd$.  We denote by $\dd\left( x, \left[ F \right]\right)$ the image of $\dd$ by $\imath_{\left( x, \left[ F \right]\right)}$ and we call it the \textit{Teichmüller disc} associated to $\left( x, \left[ F \right] \right)$. Note that the notion of Teichmüller disc already appeared in the most famous Teichmüller paper \cite{teichmullermain} under the name ``complex geodesic'' (see §161).  As the upper half-plane is biholomorphic to the unit disc, we shall consider $\mathbb{H}$ instead of $\dd$.

There exists another point of view on the Teichmüller disc which is more geometric. The point $x\in\tei$ is determined by the transverse pair $\left( f_{\ast}\left( F\right), F_{v, q}\right)$. Such a pair gives a system of coordinates which are natural coordinates for $q$. An element of $\quot{\sl2}{\so2}$ acts on such a coordinates and defines a new transverse pair of measured foliation and so a new point in the Teichmüller space. Furthermore, $\quot{\sl2}{\so2}$ is isomorphic to the upper half plane  and the orbit of $x$ by this group is the Teichmüller disc $\dd\left( x, \left[F \right] \right)$. For more details, we refer to \cite{schmithusen&herrlich}. 

We  deduce from this second point of view the following elementary result.
\begin{lemme}\label{lemmedisque}
Let $x$, $y \in\tei$ and $F\in\mf$. If $y\in \dd\left( x, \left[ F\right]\right)$, then $\dd\left( x, \left[ F\right]\right)$ and $\dd\left( y, \left[ F\right]\right)$ are identical up to an automorphism of the disc.
\end{lemme}

Even if this result is well known,  we sketch a proof.

\begin{proof}
As $y\in\dd\left( x, \left[ F \right] \right)$, there exists a pair $\left( s, t \right)\in\mathbb{R}^2$ such that $y$ is determined by 
\begin{equation}\label{matricedisque}
\begin{pmatrix}
1 & s \\ 0 & 1
\end{pmatrix}
\cdot \begin{pmatrix}
e^{-\frac{t}{2}} & 0 \\ 0 & e^{\frac{t}{2}}
\end{pmatrix}.
\end{equation}
These two matrices which act on natural coordinates, preserve (projectively) the measured foliation $F$, and so they determine a regular quadratic differential on $y$ whose the horizontal foliation is projectively the same as $F$. Using the inverse matrix of (\ref{matricedisque}) on these new natural coordinates, we obtain $x$ and so $x\in\dd\left( y , \left[ F \right] \right)$ and by Lemma 2.1 of \cite{marden&masur} we complete the proof. 
\end{proof}

From the proof of Lemma \ref{lemmedisque}, we observe  that for any $t\in\mathbb{R}$, $\ray_{\left[ F \right]}^{t}\left( \cdot \right)$ is identified with the diagonal matrix of (\ref{matricedisque}), and so it  preserves $\dd\left( x, \left[F \right]\right)$. Thus, by pulling back the Teichmüller disc to $\mathbb{H}$, we can consider this  Teichmüller ray as a map from $\mathbb{H}$ to $\mathbb{H}$ such that for any $t\in\mathbb{R}$ and any $z= x+ \ii y \in\mathbb{H}$,
\begin{equation}\label{expressionray}
\ray_{\left[ F \right]}^{t}\left( z \right)= x+\ii e^{t}y.
\end{equation}

The parabolic element in (\ref{matricedisque}) corresponds up to normalization to what we shall call the horocyclic deformation directed by $F$. A study of such a deformations in the Teichmüller space is done in the rest of this note.

\subsection{Horocyclic deformation} 

We define the \textit{horocyclic deformation} as follows.
\begin{definition}\label{horo}
Let $t\in\mathbb{R}$ and $F\in\mf$. The horocyclic deformation directed by $F$ of parameter $t$ is
\begin{align*}
\horo_{\left[ F\right]}^{t} \, :\, \tei &\rightarrow \tei  \\
x &\mapsto \imath_{\left( x, \left[ F \right]\right)} \left( k_t e^{\ii\theta_t}\right),
\end{align*}
where $k_t =\frac{1}{\sqrt{1+\frac{4\ext_{x_0}\left( F \right)^2}{t^2 \ext_{x}\left( F \right)^2}}}$ and $\theta_t =\arctan\left( \frac{2\ext_{x_0}\left( F \right)}{t\ext_{x}\left( F \right)}\right)$.
\end{definition}

We observe that for a fixed real number $t$, the  horocyclic deformation depends only on the projective class of the given  measured foliation. Thus, we can suppose that the foliation $F$ belongs to $\mf_1$. 

As for Teichmüller rays, by pulling back $\dd\left( x, \left[ F \right]\right)$ to $\mathbb{H}$, one can check with our normalization that for any $s\in\mathbb{R}$,
\begin{equation}\label{expressionhoro}
\horo_{\left[ F \right]}^{s}\left( \ii \right)= \ii -s\cdot \ext_{x}\left( F \right).
\end{equation}
Thus, the image  of $t\in\mathbb{R}\mapsto\horo_{\left[ F \right]}^{t}\left( x \right)$ coincides with the image by $\imath_{\left( x, \left[F \right] \right)}$ of a certain horocycle.

Moreover, as for any point in $\dd\left( x, \left[ F \right]\right)$, the Teichmüller ray at this point directed by $F$ stays in this disc, we  deduce  that for any $s\in\mathbb{R}$, $\horo_{\left[ F \right]}^{t}\left( \cdot \right)$ preserves the associated Teichmüller disc. 

As for the Teichmüller line, we can give an explicit expression of the action on the upper half-plane by  the horocyclic deformation. It suffices to conjugate it by an appropriate automorphism in order to bring back the problem  in $\ii$. We can also deduce from Relation (\ref{expressionhoro}) (and even from the definition) that for any point $x\in\tei$ and for any $F\in\mf_1$ the map $s\in\mathbb{R}\mapsto \horo_{\left[ F \right]}^{s}\left( x \right)$ is continuous. 

In the case where $F$ is a simple closed curve $\alpha$, it is important to note that for any point $x\in\tei$, $\left( \horo_{\left[ \alpha\right]}^{n}\left( x \right)\right)_{n\in\mathbb{Z}}$ corresponds to the orbit of $x$ by the group generated by the Dehn twist along $\alpha$. Such a Dehn twist is denoted by $\tau_{\alpha}$. This fact  was  observed by Marden and Masur in \cite{marden&masur} and it will be used below. Marden and Masur also gave  a description of $\horo_{\left[ \alpha\right]}^{t}\left( x \right)$ when $t$ is real. Such points are described by what we call \emph{conformal twist}  along $\alpha$ of parameter $t$.

As we shall see below, the horocycle deformation has some similarities with the earthquake map.

\subsection{Elementary properties}

The first property which can be seen as an analogue of a  theorem of Thurston (see Theorem 2 in \cite{kerckhoffnielsen} for the statement and a proof) is the existence of a certain horocyclic deformation between any two points of the Teichmüller space. The statement is the following.

\begin{prop}
Let $x$ and $y$ be two distinct points in $\tei$. Then there exists a unique $F\in\mf_1$ and a unique $s>0$ such that
$$
y= \horo_{\left[ F \right]}^{s}\left( x \right).
$$
\end{prop}

\begin{proof}
By Theorem \ref{theoremteich},  there exists a unique $G\in\mf_1$ and a unique $s>0$ such that $y=\ray_{\left[ G \right]}^{s}\left( x \right)$. Thus, we just have to consider $e^{-\ii \tau} \cdot q_{G}$ for a some $\tau$ and set $F=F_{v,e^{-\ii \tau}q_{G} }$. 
\end{proof}

As the horocyclic deformation and the Teichmüller deformation preserve the Teichmüller disc for a given foliation, we can state some elementary results. 

\begin{prop}\label{commute}
Let $F\in\mf$. Then for any $s\in\mathbb{R}$ and any $t\in\mathbb{R}$ we have
$$
\horo_{\left[ F \right]}^{s}\circ\ray_{\left[ F \right]}^{t}=\ray_{\left[ F \right]}^{t}\circ \horo_{\left[ F \right]}^{s}.
$$
\end{prop}

\begin{proof}
Let $x\in\tei$. We fix $\left( s, t \right)\in\mathbb{R}^2$.  As the transformations that we consider preserve the Teichmüller disc $\dd\left( x, \left[ F \right]\right)$, we will do computations in the upper half-plane. Thus, $x$ corresponds to $\ii$. From Relations (\ref{expressionray}) and (\ref{expressionhoro}), we get
$$
\ray_{\left[ F \right]}^{t}\left(\horo_{\left[ F \right]}^{s}\left( \ii \right)\right)=-s\cdot\ext_{x}\left( F \right) + \ii\cdot e^{t}.
$$
On the other hand, using Relation (\ref{rel}) and conjugating the horocyclic deformation of $\ray_{\left[ F \right]}^{t}\left( \ii \right)$ by $z\mapsto e^{-t}\cdot z$, we get
\begin{align*}
\horo_{\left[ F \right]}^{s}\left( \ray_{\left[ F \right]}^{t}\left( \ii \right) \right) &= e^{t}\cdot \left( \ii - s\cdot\ext_{\ray_{\left[ F \right]}^{t}\left( x \right)}\left( F \right) \right) \\
&=  - s \cdot \ext_{x}\left( F \right) + \ii e^{t}.
\end{align*}
The proof is complete.
\end{proof}

Note that this result is analogous to a result of Théret in \cite{theret} where he proves that the operation of earthquake and that  of stretching commute if their directions are the same. 

\begin{remarque}
If the directions for the Teichmüller deformation and for the horocyclic deformation are not the same, we do not have necessarily Property \ref{commute}. Indeed, let $\alpha$ and $\beta$ be two distinct simple closed curves  such that $i\left( \alpha , \beta\right)\neq 0$. Assume that for any $s, t \in\mathbb{R}$, 
$$
\ray_{\left[ \alpha \right]}^{t}\circ \horo_{\left[ \beta\right]}^{s}=\horo_{\left[\beta\right]}^{s}\circ\ray_{\left[\alpha\right]}^{t}.
$$
In particular this is true when $s=1$. As we said above, the horocyclic deformation of parameter $1$ corresponds to the Dehn twist along $\beta$. If we fix a point $x\in\tei$, we get for any $t\geq 0$
\begin{equation}\label{contre}
 \ray_{\left[\alpha\right]}^{t}\left( \tau_{\beta} \cdot x \right) = \tau_{\beta}\cdot \ray_{\left[\alpha\right]}^{t}\left( x\right).
\end{equation}
We recall that for any  $y=\left[ Y, g \right]\in\tei$, $\tau_{\beta} \cdot y = \left[ Y, g\circ \tau_{\beta}^{-1}\right]$. However, Gardiner and Masur showed in \cite{gm} that for any $y\in\tei$, $\ray_{\left[\alpha\right]}^{t}\left( y\right)\underset{t\rightarrow +\infty}{\cvggm} \left[ \alpha \right]$ and Miyachi proved in \cite{miyachi1} that the \emph{mapping class group} extends to the Gardiner-Masur boundary. Thus,  when $t$ tends to $+\infty$ in  Equality (\ref{contre}),  we obtain
$$
\left[ \alpha\right]= \tau_{\beta}\cdot \left[ \alpha \right],
$$  
which is obviously not true. 

Because the Teichmüller rays directed by simple closed curves converge in the Thurston boundary, we may do the same reasoning by using the convergence in $\pmf$. 
\end{remarque}

Another interesting fact is that the horocyclic deformation is continuous with respect to the direction. The statement is the following.

\begin{lemme}\label{continue}
Let $x=\left[ X, f \right]\in\tei$ and $s\in\mathbb{R}$. Then $\horo_{\left[ \cdot \right]}^{s}\left( x \right) : \pmf \rightarrow \tei$ is continuous. 
\end{lemme}

\begin{proof}
This is just a consequence of Theorem \ref{theoremteich}. Indeed, let $ \left[ F_n \right] $ be a sequence of elements in $\pmf$ which converges to $\left[ F \right]$ for the topology induced by the geometric intersection. Let $\left( q_n \right)_n$ and $q_{F}$ be the corresponding elements in $\qd_{1}\left( X \right)$. For any $n$, the point $\horo_{\left[ F_n \right]}^{s}\left( x \right)$  is determined by the teichmüller deformation of parameter $2\tanh^{-1}\left( k_s \right)$ directed by the horizontal foliation of $e^{-\ii \theta_{s}}\cdot q_{n}$. We have the same description for $\horo_{\left[ F \right]}^{s}\left( x \right)$ by using $q_F$ instead of $q_n$. As the map given in Theorem \ref{theoremteich} is a homeomorphism, a fortiori it is continuous with respect to $\pmf$ and so the lemma is proved. 
\end{proof}

This lemma will be useful to prove that the extremal length of a particular foliation does not change along a horocyclic deformation. Indeed, we have

\begin{prop}\label{invariance}
Let $F\in\mf$ and $x=\left[ X, f \right]\in\tei$. Then 
$$
\forall s\in\mathbb{R}, \, \ext_{\horo_{\left[ F \right]}^{s}\left( x\right)}\left( F \right)=\ext_{x}\left( F \right).
$$
\end{prop}

\begin{proof}
We will start by showing this property in the case of simple closed curves and then, by using Lemma \ref{continue} and the fact that $\mathcal{S}$ is dense in $\pmf$  we will get the general case. 

Let $\alpha\in\mathcal{S}$ and $q_\alpha$ the corresponding quadratic differential on $X$. We recall that the horizontal foliation of $q_\alpha$ is exactly $f_{\ast}\left( \alpha\right)$. Thus, the complement in $X$ of the corresponding critical graph is biholomorphic to a cylinder $A$ of modulus $M=\frac{1}{\ext_{x}\left( \alpha \right)}$. We can consider that $A$ is the planar annulus of inner radius $1$ and outer radius $\exp\left( 2\pi M\right)$.  Now we fix $s\in\mathbb{R}$ and we denote by $x_s$ the image of $x$ by $\horo_{\left[\alpha \right]}^{s}$. Following \cite{marden&masur}, $x_s =\left[ X_s , f_s \right]$, where $f_s$ is the quasiconformal map which lifts to $\tilde{f}_s : A \rightarrow A$; $z\mapsto z\vert z \vert^{\ii\frac{s}{M}}$. We say that the map $f_s$ is the conformal twist along $\alpha$ of parameter $t$.  The surface $X_s$ is obtained by identifying some parts of boundary components of $A$ and so,  if $s$ is not an integer, then $X_s$ is different to $X$. As the map $\tilde{f}_{s}$ does not change the modulus of $A$ and preserves the core curve which is in the class of $\alpha$, we deduce from the  geometric definition of extremal length that 
$$
\ext_{x_s}\left( \alpha\right)= \ext_{x}\left( \alpha\right).
$$ 
\end{proof}

We  deduce the following result.

\begin{corollaire}
Let $F\in\mf$ and $x\in\tei$. Then for any $s$, $t\in\mathbb{R}$
$$
\horo_{\left[ F \right]}^{s+t}\left( x \right)= \horo_{\left[ F \right]}^{s}\left( \horo_{\left[ F \right]}^{t}\left( x \right)\right).
$$
In particular, $\horo_{\left[ F \right]}^{t}\, : \tei \rightarrow \tei$ is a bijection.
\end{corollaire}

\begin{proof}
By (\ref{expressionhoro}) we have 
$$
\horo_{\left[ F \right]}^{s+t}\left( x \right)= \ii - \left( s+ t \right)\cdot \ext_{x}\left( F \right).
$$
On the other hand, using Property \ref{invariance} we get 
$$
\horo_{\left[ F \right]}^{s}\left( \horo_{\left[ F \right]}^{t}\left( x \right)\right) = \ii -\left( s+ t \right)\ext_{x}\left( F \right).
$$
\end{proof}

\begin{question}
Does $\horo_{\left[ F \right]}^{s}\left( \cdot \right)$  is a homeomorphism?
\end{question}

The main interest of this paper concerns the asymptotic behaviour of the horocyclic deformation in the Gardiner-Masur boundary. This is explained by the following question.

\begin{question}\label{questioncvg}
Does $\left( \mathcal{H}^{t}_{\left[ F\right]}\left( x \right) \right)_{t}$ converge in the Gardiner-Masur boundary when $t\rightarrow \pm\infty$?
\end{question}

Looking at Figure \ref{horocycle}, a naive guess would be that $\left( \mathcal{H}^{t}_{\left[ F \right]}\left( x \right) \right)_{t}$ converges and the limit would be exactly  the limit of the Teichmüller ray determined by $F$. This is the case if $\dim_{\mathbb{C}}\tei =1$. Indeed, the embedding (\ref{teichmullerdisc}) is an homeomorphism and from \cite{miyachiproc}, we have that this homeomorphism can be continuously extended to the boundary. Unfortunately, Miyachi proved in \cite{miyachilip} (Subsection 8.1) that if $\dim_{\mathbb{C}}\tei \geq 2$,  the embedding $\imath_{\left( x, \left[ F \right]\right)}$ does not extend continuously to the Gardiner-Masur boundary. However, as we shall see below, the result holds in at least two particular cases.

\begin{figure}[ht]
\begin{center}
\psfrag{x}{$x$}
\psfrag{D}{$\dd\left( x , \left[ F \right]\right)$}
\psfrag{tei}{$\ray_{\left[ F \right]}^{t}\left( x \right)$}
\psfrag{horo}{$\horo_{\left[ F \right]}^{t}\left( x \right)$}
\includegraphics[width=0.6\linewidth]{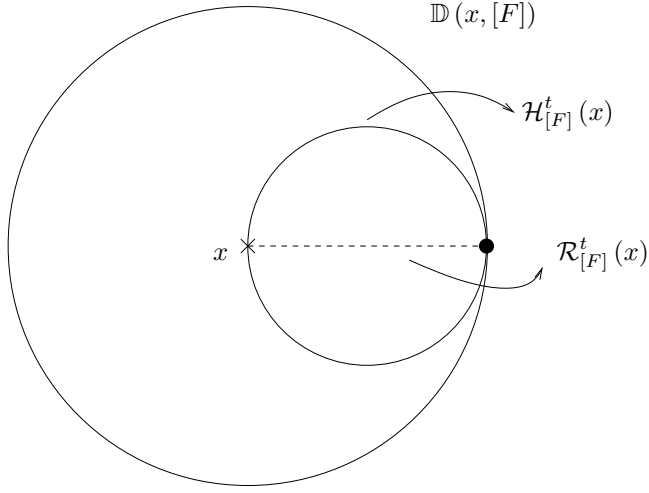}
\caption{The Teichmüller disc $\dd\left( x , \left[ F \right]\right)$. Points of the circle which crosses $x$ are horocyclic deformations directed by $\left[ F \right]$ and points of the dotted segment are Teichmüller deformations directed by $\left[ F \right]$.}\label{horocycle}
\end{center}
\end{figure}

\section{Convergence in the Gardiner-Masur boundary}

\subsection{The simple closed curves case}

Let $x=\left[ X, f \right]\in\tei$ and $\alpha\in\mathcal{S}$. We are interested in the convergence of $\left( \horo_{\left[\alpha\right]}^t \left( x \right)\right)_{t}$ when $t\rightarrow \pm\infty$. We recall that $\tau_{\alpha}$ denotes the Dehn twist along $\alpha$. To simplify notation we set for any $t\in\mathbb{R}$, $x_t = \horo_{\left[ \alpha \right]}^{t}\left( x \right)$. As we remarked in the proof of Property \ref{invariance}, $x_t =\left[ X_t , f_t \right]$ where $f_t$ is the  conformal twist along $\alpha$ of parameter $t$. We also recall that if $t\in\mathbb{Z}$, then $f_t = f\circ \tau_{\alpha}^{-t}$.

In order to study the convergence of $\left( x_t \right)_t$, we will have to use another result of  Miyachi.
\begin{theo}[\cite{miyachi2}, Theorem 3]\label{caraczero}
Let $F\in\mf$ be either a uniquely ergodic foliation or a simple closed curve. Let $p\in\teigmb$. If for all $G\in\mf$ such that $i\left( F, G\right)=0$ we have $\epsilon_{p}\left( G \right)=0$, then
$$
\epsilon_p \left( \cdot \right) =\frac{1}{\ext_{x_0}^{\frac{1}{2}}\left( F \right)}\cdot i\left( F, \cdot \right).
$$
\end{theo}

We have all the  elements to establish the following result.

\begin{theo}\label{theocourbe}
With the above notation,  
$$x_t \underset{t\rightarrow \pm \infty}{\cvggm} \left[ \alpha \right].$$
\end{theo}

\begin{proof}
Let $p\in\teigmb$ be any cluster point of $\left( x_t \right)_t$. Up to subsequence, we can assume that $x_t \underset{t\rightarrow\infty}{\cvggm} p$. By Property \ref{invariance}, we already have
\begin{align*}
\epsilon_{p}\left( \alpha \right) &= \lim_{t\rightarrow \infty}\epsilon_{x_t} \\
&=  \lim_{t\rightarrow \infty}\left( \frac{\ext_{x_t}\left( \alpha \right)}{e^{2\cdot\tanh^{-1}\left( k_t \right)}}\right)^{\frac{1}{2}} \\
&= \lim_{t\rightarrow \infty}\left( \frac{\ext_{x}\left( \alpha \right)}{e^{2\cdot\tanh^{-1}\left( k_t \right)}}\right)^{\frac{1}{2}} \\
&= 0. 
\end{align*}
Now, let  $\beta\in\mathcal{S}$ such that $i\left( \alpha, \beta\right)=0$. For any $t\in\mathbb{R}$, we have, by the quasiconformal distorsion (or the Kerckhoff formula),
\begin{align*}
\ext_{x_t}\left( \beta\right) &\leq e^{d_{T}\left( x_{\lfloor t \rfloor} , x_t \right)}\cdot \ext_{x_{\lfloor t \rfloor}}\left( \beta \right) \\
&\leq e^{d_{T}\left( x_{\lfloor t \rfloor} , x_{\lceil t \rceil} \right)}\cdot \ext_{X}\left( \tau_{\alpha}^{-\lfloor t \rfloor}\left( \beta \right) \right) \\
&\leq  e^{d_{T}\left( x_{\lfloor t \rfloor} , x_{\lceil t \rceil} \right)}\cdot \ext_{x}\left(  \beta  \right).
\end{align*}
Furthermore, the mapping class group acts by isometries with respect to the Teichmüller distance, then $d_{T}\left( x_{\lfloor t \rfloor} , x_{\lceil t \rceil} \right)=d_{T}\left( x, x_1 \right)$. Thus, $\left( \ext_{x_t}\left( \beta \right)\right)_t$ is bounded from above and we deduce that 
$$
\epsilon_{p}\left( \beta \right)= \lim_{t\rightarrow \infty}\left( \frac{\ext_{x_t}\left( \alpha \right)}{e^{2\cdot\tanh^{-1}\left( k_t \right)}}\right)^{\frac{1}{2}} =0.
$$
The conclusion follows by using Theorem \ref{caraczero}.
\end{proof}

This result is analogous to the convergence of the Frenchel-Nielsen deformations in the Thurston boundary. Indeed, it is well known that a Frenchel-Nielsen deformation determined by a simple closed curve, converges to this simple closed curve. By the way, from the proof of Theorem \ref{theocourbe}, we can deduce the following corollary.

\begin{corollaire}
Let $\alpha$ be a simple closed curve. Let $x\in\tei$. Then 
$$
\horo_{\left[ \alpha \right]}\left( x\right) \underset{t\rightarrow \pm\infty}{\cvgth} \left[ \alpha\right].
$$
\end{corollaire}

To prove this corollary we need to recall some facts. We recall that the \emph{Thurston asymmetrimetric metric} $d_{\textrm{Th}}\left( \cdot , \cdot \right)$ can be defined as follows.
$$
\forall x, y\in\tei, \; d_{\textrm{Th}}\left( x,y \right)=\log\sup_{\alpha\in\mathcal{S}}\frac{l_{y}\left( \alpha\right)}{l_{x}\left( \alpha\right)}.
$$
This metric was introduced by Thurston in \cite{thurston} and some investigations about it can be found in \cite{papadop&theret}, \cite{theret} and \cite{liu&papadop&su&theret}. We refer also to \cite{papadop&theret1}. We can recognize some similarities with the Kerckhoff formula. Furthermore, by setting for any $x\in\tei$,
\begin{align*}
\mathcal{L}_{x} : \mf &\rightarrow \mathbb{R}_{+} \\
 F & \mapsto \frac{l_{x}\left( F \right)}{e^{d_{\textrm{Th}}\left( x_0 , x \right)}},
\end{align*}
Walsh proved in \cite{walsh} that a sequence $x_n$ in the Teichmüller space converges to the projective class of $G$ in the Thurston boundary, if and only if, $\mathcal{L}_{x_n}$ converges to $F\in\mf \mapsto C\cdot i\left( G, F\right)$ uniformly on compact sets of $\mf$. The constant $C$  depends on $x_0$ and $G$.   

\begin{proof}
Let us denote by $\left( x_t \right)_{t}$ the sequence $\left( \horo_{\left[ \alpha \right]}\left( x\right) \right)_{t}$. Let $\left[ G \right]\in\pmf$ be any cluster point of $\left( x_t \right)_t$. By the analytic definition of extremal length and the Gauss-Bonnet formula, we have that for any $\beta \in\mathcal{S}$ and any $t\in\mathbb{R}$,
\begin{equation}\label{comphypext}
l_{x_t}^2 \left( \beta\right)\leq 2\pi\vert \chi\left( X_0 \right) \vert \ext_{x_t}\left( \beta \right).  
\end{equation}
Thus, for any $\beta\in\mathcal{S}$ such that $i\left( \alpha , \beta \right)=0$, we know from the proof of Theorem \ref{theocourbe} that $\left(\ext_{x_t}\left( \beta\right)\right)_{t}$ is bounded from above, and so from (\ref{comphypext}) we have that $\left( l_{x_t}\left( \beta \right) \right)_t$ is also  bounded from above. Then, we deduce that
$$
\frac{l_{x_t}\left( \beta \right)}{e^{d_{\textrm{Th}}\left( x , x_t \right)}}\underset{t\rightarrow \pm\infty}{\rightarrow} 0,
$$ 
and by the result of Walsh, we can say that 
$$
i\left( G , \beta\right)=0. 
$$
As this equality is true for any simple closed curve whose its geometric intersection with $\alpha$ is zero, we deduce that $G$ is topologically the same foliation as $\alpha$ and so $G$ is projectively equivalent to $\alpha$. This  fact  is true for any cluster point of $x_t$ and then the proof is done.
\end{proof}

\begin{remarque}
In contrast to the uniquely ergodic case, the author does not know if  for a sequence $x_n$ in the Teichmüller space and a simple closed curve $\alpha$ we have
$$
x_n \underset{n\rightarrow +\infty}{\cvggm}\left[ \alpha \right] \Leftrightarrow x_n \underset{n\rightarrow +\infty}{\cvgth}\left[ \alpha \right].
$$
We have this property only in few cases, as when the sequence is given by the Teichmüller deformation or the horocyclic deformation directed by a simple closed curve.  These examples are from the conformal point of view of the Teichmuller space, but we can show that it is also true for the Fenchel-Nielsen deformation and for the stretch lines when the associated horocyclic foliation is a simple closed curve.
\end{remarque}

From Theorem \ref{theocourbe}, we also deduce that we can find two sequences $y_n$ and $z_n$ which are at the same distance from $x$ and converge to the same point in the Gardiner-Masur boundary, but with $d_T \left( y_n , z_n \right) \rightarrow +\infty$. It suffices to take $y_n = \horo_{\left[ \alpha \right]}^{n}\left( x \right)$ and $x_n = \ray_{\left[ \alpha \right]}^{2\cdot \tanh^{-1}\left( k_n \right)}\left( x \right)$. We thus obtain  a positive answer to Question \ref{question1}.

\subsection{The uniquely ergodic case}

Let $x\in \tei$ and $F\in\mf$ be a uniquely measured foliation. 
Before studying the asymptotic behaviour of $\left( \horo_{\left[ F \right]}^{t}\left( x \right)\right)_t$ we recall some facts about the \textit{Gromov product}. The Gromov product of $y$ and $z$ with basepoint $x$ for $d_T$ is defined by
$$
\langle y\mid z \rangle_{x}= \frac{1}{2}\left( d_T \left( x, y\right) + d_T \left(x, z \right) - d_T \left( y , z \right) \right).
$$
Miyachi proved in  \cite{miyachigromov}, that the Gromov product at $x$ has a continuous extension to $\teigm \times \teigm$ with value in $\left[0 , +\infty \right]$. He also gave an explicit expression in terms of extremal length. For any $p\in\teigmb$, the Gromov product of $p$ and $\left[ F \right]$ is 
\begin{equation}\label{gromov}
\langle p \mid \left[F\right] \rangle_{x} = -\frac{1}{2}\log\left( \frac{\epsilon_{p}\left( F\right)}{\ext_{x}^{\frac{1}{2}}\left( F \right)}\right).
\end{equation}
If we set for all $t\in\mathbb{R}$, $y_t =\horo_{\left[ F\right]}^{t}\left( x \right)$ and $z_t = \ray_{\left[ F \right]}^{\vert t \vert}\left( x \right)$, we have by the embedding of the disc into  Teichmüller space, the following:
$$
\langle y_t \mid z_t \rangle_{x} \underset{t\rightarrow \infty}{\longrightarrow}  +\infty.
$$
Then, for any cluster point $p$ of $y_t$ we have, using Relation (\ref{gromov}),  $\epsilon_p \left( F \right)=0$. By Theorem \ref{caraczero} and the fact that $F$ is uniquely ergodic, we conclude that $p=\left[ F \right]$. Thus, we have proved:

\begin{theo}\label{theoergodic}
The horocylcic deformation diricted by a uniquely ergodic foliation $F$ converges in the Gardiner-Masur boundary to the associated projective foliation. 
\end{theo}

Following results of Miyachi (Corollary 5.1 in \cite{miyachi1} and Theorem \ref{caraczero}, we have that a sequence $y_n$ in Teichmüler space converges to a class of uniquely ergodic foliation with respect to the Thurston embedding if and only if it converges to the same class of foliation with respect to the  Gardiner-Masur embedding. Thus we have

\begin{corollaire}
Let $F$ be a uniquely ergodic foliation and $x$ be a point in $\tei$. Then 
$$
\horo_{\left[ F \right]}^{t}\left( x\right) \underset{t\rightarrow \pm\infty}{\cvgth} \left[ F \right].
$$
\end{corollaire} 

We have seen that the horocyclic deformation directed by a simple closed curve or a uniquely ergodic foliation converges in the Gardiner-Masur boundary and the limit is exactly the same limit as the Teichmüller ray directed by the same foliation. Thus, in these two particular cases we have given a positive answer to Question \ref{questioncvg}. 
In the most general case, we can except that horocyclic deformation for a given direction converges to the same limit as the Teichmüller ray with the same direction. A negative result is given below and in some sense was already observed by Gardiner and Masur in \cite{gm} to prove that $\pmf$ is strictly contained in the Gardiner-Masur.

\subsection{An example of rational foliation}

\begin{figure}[ht]
\begin{center}
\psfrag{X}{$X_{0}^{d}$}
\psfrag{a1}{$\alpha_1$}
\psfrag{a2}{$\alpha_2$}
\psfrag{d}{$\delta$}
\psfrag{G}{$\Gamma$}
\psfrag{b1}{$\beta_1$}
\psfrag{b2}{$\beta_2$}
\includegraphics[width=0.9\linewidth]{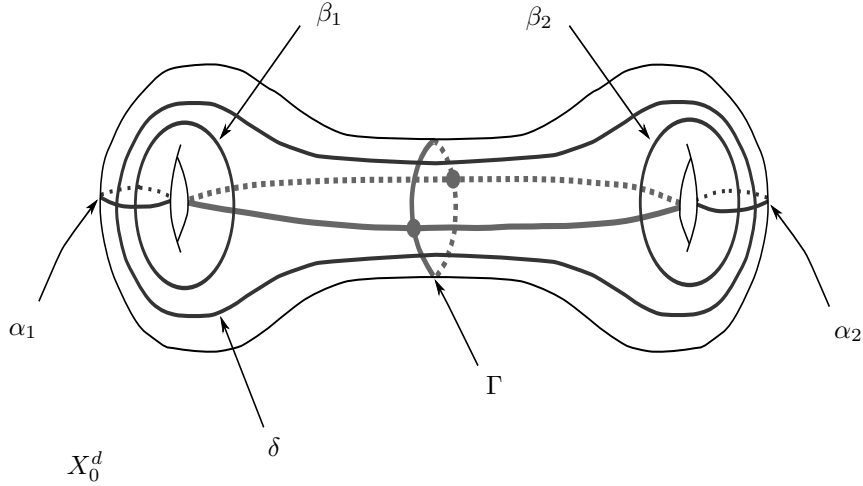}
\caption{On the symmetric Riemann surface $X_{0}^{d}$, we draw the critical graph $\Gamma$ of the measured foliation $F=\alpha_1 + \alpha_2$. Up to a Withehead move, we can assume that this measured foliation has two critical points of order 4.}\label{contreexemple}
\end{center}
\end{figure}

Let $X_0 ^d$ be a Riemann surface of genus $2$ obtained by gluing two  tori  with one boundary component along their boundary,  one of them being the mirror conformal structure of the other. We denote them by $T$ and $\overline{T}$. Thus, we get a natural anti-holomorphic involution on $X_0 ^d$ which can be seen as  a complex conjugation.  We denote it by $i_{X_0 ^d}$. We fix $\alpha_1$ and $\alpha_2$, two disjoints simple closed curves as in Figure \ref{contreexemple} such that $\alpha_2$ is obtained by conjugating $\alpha_1$. To be more precise, $\alpha_2 = i_{X}\left( \alpha_1 \right)$. Up to changing $T$, we can assume without loss of generality that
\begin{equation}\label{modulesym}
\ext_{T}\left( \alpha_{1}\right) = \ext_{\overline{T}}\left( \alpha_2 \right) = 1.
\end{equation}

We set $F= \alpha_1 + \alpha_2$ and $x_0 ^d=\left[ X_0 ^d, \id\right]$. Thus, by symmetry, we deduce that the quadratic differential $q_F$ is invariant by $i_{X_0 ^d}$ and that
\begin{equation}\label{modulesym2}
\ext_{x_0}\left( F \right)= \| q_F \|=2.
\end{equation}

By works of Marden and Masur in Section 2 of \cite{marden&masur},  we deduce that
\begin{equation}
\forall n\in\mathbb{N}, \,\horo_{\left[F \right]}^{2n} \left( x_0 ^d \right) = \left( \tau_{\alpha_1}^{n} \circ \tau_{\alpha_2}^{n}\right) \cdot x_0 ^d.
\end{equation}
To simplify notation, we set for any $n\in\mathbb{Z}$,  $x_n = \horo_{\left[F \right]}^{2n} \left( x_0 ^d \right)$.

By Kerckhoff's computation in \cite{kerckhoff} (see also the appendix of \cite{miyachi1}, for more details) and Equality (\ref{modulesym}) we  conclude that 
\begin{equation}\label{limitrationnel}
\ray_{\left[ F\right]}^{t}\left( x_0 ^d \right) \underset{t\rightarrow+\infty}{\cvggm} \left[ \left(  i\left( \alpha_1 , \cdot\right)^2 +  i\left( \alpha_2 , \cdot\right)^2   \right)^{\frac{1}{2}} \right] \in\prs.
\end{equation}

 As Miyachi proved in \cite{miyachi1}, this limit is not an element of $\pmf$. 

Following \cite{gm}, let us show that any cluster points  of $\left( x_n \right)$ in $\teigmb$ is different of the limit in (\ref{limitrationnel}). Let $q\in\teigmb$ be any cluster point of $\left( x_n \right)$. Assume that $q$ is equal to the limit of the Teichmüller ray directed by $F$. Then there exists $\lambda>0$ such that
\begin{equation}\label{absurde}
\forall \gamma \in\mathcal{S}, \,\epsilon_{q}\left( \gamma \right)^2 = \lambda \cdot \left(   i\left( \alpha_1 , \gamma \right)^2 +  i\left( \alpha_2 , \gamma\right)^2 \right).  
\end{equation}
However, if $\beta_1$, $\beta_2$ and $\delta$ are like in Figure \ref{contreexemple}, then for any integer $n$
\begin{align*}
\begin{cases}
\epsilon_{x_n}\left( \beta_1 \right)^2 &= \frac{1}{4 n^2}\cdot \ext_{X_0 ^d}\left( \tau_{\alpha_1}^{-n}\left( \beta_1 \right) \right), \\
\epsilon_{x_n}\left( \beta_2 \right)^2 &= \frac{1}{4 n^2}\cdot \ext_{X_0 ^d}\left( \tau_{\alpha_2}^{-n}\left( \beta_2 \right) \right), \\
\epsilon_{x_n}\left( \delta \right)^2 &= \frac{1}{4 n^2}\cdot \ext_{X_0 ^d}\left( \left(\tau_{\alpha_1}^{-n}\circ\tau_{\alpha_2}^{-n}\right)\left( \delta \right) \right).
\end{cases}
\end{align*}
Moreover, as for  $\frac{1}{n}\cdot \tau_{\alpha_i}^{-n}\left( \beta_i \right)$ tends to $\alpha_i$ (for $i=1, \, 2$) and $\frac{1}{n}\cdot \left(\tau_{\alpha_1}^{-n}\circ\tau_{\alpha_2}^{-n}\right)\left( \delta \right)$ tends to $F$ when $n\rightarrow \infty$, we deduce by continuity of the extremal length that 
\begin{align*}
\begin{cases}
\epsilon_{q}\left( \beta_1 \right)^2 &=\frac{1}{4}\cdot \ext_{X_0 ^d}\left( \alpha_1 \right), \\
\epsilon_{q}\left( \beta_2 \right)^2 &=\frac{1}{4}\cdot \ext_{X_0 ^d}\left( \alpha_2 \right), \\
\epsilon_{q}\left( \delta \right)^2 &=\frac{1}{4}\cdot \ext_{X_0 ^d}\left( F \right)=\frac{1}{2}.
\end{cases}
\end{align*}
By symmetry, we have $\ext_{X_0 ^d}\left( \alpha_1 \right)=\ext_{X_0 ^d}\left( \alpha_1 \right)=c$ and by the geometric definition of extremal length we have necessarily $c<1$. Thus, by comparing with Relation (\ref{absurde}), we get a contradiction.

We just saw that the horocyclic deformation directed by $F$ cannot converge to the same limit as the Teichmüller ray directed by the same foliation. However, the author does not know if in this particular  case, the horocyclic deformation converges in the Gardiner-Masur boundary.

\bibliographystyle{acm}
\bibliography{biblio}

\medskip
\medskip

\textsc{Vincent Alberge, Université de Strasbourg, 7 rue René Descartes, 67084 Strasbourg Cedex, France}

\textit{E-mail address:} { \href{mailto:alberge@math.unistra.fr}{\tt alberge@math.unistra.fr} } 

\end{document}